\documentclass[final]{amsart}   

\usepackage{amsmath,amssymb,latexsym, amscd}
\usepackage{exscale, cite, eps fig, graphics}

\usepackage[notref, notcite]{showkeys}
\usepackage{showkeys}

\renewcommand{\geq}{\geqslant}
\renewcommand{\leq}{\leqslant}

\renewcommand{\l}{\langle}
\renewcommand{\r}{\rangle}

\renewcommand{\L}{\mathcal{L}}
\newcommand{\M}{\mathcal{M}}
\renewcommand{\k}{\kappa}

\newtheorem{theorem}{Theorem}[section]
\newtheorem{lemma}[theorem]{Lemma}
\newtheorem{proposition}[theorem]{Proposition}

\newtheorem{corollary}[theorem]{Corollary}

\newtheorem*{main-theorem}{Main Theorem}
\newtheorem*{remark*}{Remark}
\numberwithin{equation}{section}

\title[Modulational instability in the Whitham equation]
{Modulational instability in\\the Whitham equation for water waves}

\author[Hur]{Vera~Mikyoung~Hur}
\address{Department of Mathematics, University of Illinois at Urbana-Champaign, Urbana, IL 61801 USA}
\email{verahur@math.uiuc.edu}

\author[Johnson]{Mathew~A.~Johnson}
\address{Department of Mathematics, University of Kansas, Lawrence, KS 66045 USA} 
\email{matjohn@math.ku.edu}  

\date{\today}

\begin{document}

\maketitle

\begin{abstract}
We show that periodic traveling waves with sufficiently small amplitudes of the Whitham equation, 
which incorporates the dispersion relation of surface water waves 
and the nonlinearity of the shallow water equations,
are spectrally unstable to long wavelengths perturbations 
if the wave number is greater than a critical value, 
bearing out the Benjamin-Feir instability of Stokes waves; 
they are spectrally stable to square integrable perturbations otherwise. 
The proof involves a spectral perturbation of the associated linearized operator 
with respect to the Floquet exponent and the small amplitude parameter. 
We extend the result to related, nonlinear dispersive equations.
\end{abstract}

\section{Introduction}\label{sec:intro}

We study the stability and instability of periodic traveling waves of the {\em Whitham equation} 
\begin{equation}\label{E:whitham1}
\partial_tu+\M \partial_xu+\frac32\sqrt{\frac{g}{d}}u\partial_xu=0,
\end{equation}
which was put forward in \cite[p. 477]{Whitham} to model the unidirectional propagation of 
surface water waves with small amplitudes, but not necessarily long wavelengths, in a channel. 
Here $t\in\mathbb{R}$ denotes the temporal variable, 
$x\in\mathbb{R}$ is the spatial variable in the predominant direction of wave propagation,
and $u=u(x,t)$ is real valued, describing the fluid surface;
$g$ denotes the constant due to gravitational acceleration and $d$ is the undisturbed fluid depth. 
Moreover $\M$ is a Fourier multiplier, defined via its symbol as 
\begin{equation}\label{def:M1}
\widehat{\M f}(\xi)=\sqrt{\frac{g\tanh(d\xi)}{\xi}}\widehat{f}(\xi)=:\alpha(\xi)\widehat{f}(\xi).
\end{equation}
Note that $\alpha(\xi)$ is the phase speed of a plane wave with the spatial frequency $\xi$ 
near the quintessential state of water; see \cite{Whitham}, for instance.
 
\

When waves are long compared to the fluid depth so that $d\xi\ll1$, 
one may expand the symbol in \eqref{def:M1} and write that 
\begin{equation}\label{def:aKdV}
\alpha(\xi)=\sqrt{gd}\Big(1-\frac16d^2\xi^2\Big)+O(d^4\xi^4)=:\alpha_{KdV}(\xi)+O(d^4\xi^4),
\end{equation}
where $\alpha_{KdV}$ gives the dispersion relation of the Korteweg-de Vries (KdV) equation, 
without normalization of parameters,
\begin{equation}\label{E:KdV}
\partial_tu+\sqrt{gd}\Big(1+\frac16d^2\partial_x^2\Big)\partial_xu
+\frac32\sqrt{\frac{g}{d}}u\partial_xu=0.
\end{equation}
Consequently one may regard the KdV equation as to approximate up to ``second" order 
the dispersion of the Whitham equation, and hence the water wave problem, 
in the long wavelength regime. 
As a matter of fact, solutions of \eqref{E:KdV} and \eqref{E:whitham1} exist 
and they converge to those of the water wave problem at the order of $O(d^2\xi^2)$ 
during a relevant interval of time; see \cite[Section~7.4.5]{Lannes}, for instance. 

The KdV equation well explains long wave phenomena in a channel of water
--- most notably, solitary waves --- but it loses relevances\footnote{
A relative error $(\frac{\alpha-\alpha_{KdV}}{\alpha})(\xi)$ of, say, $10\%$ 
is made for $|d\xi|>1.242\dots$.} for short and intermediately long waves. 
Note in passing that a long wave propagating over an obstacle releases higher harmonics;
see \cite[Section~5.2.3]{Lannes}, for instance.
In particular, waves in shallow water at times develop a vertical slope or a multi-valued profile
whereas the KdV equation prevents singularity formation from solutions. 
Whitham therefore concocted \eqref{E:whitham1} as an alternative to the KdV equation, 
incorporating the full range of the dispersion of surface water waves  
(rather than a second order approximation)
and the nonlinearity which compels blowup in the shallow water equations. 
As a matter of fact, \eqref{E:whitham1} better approximates 
short-wavelength, periodic traveling waves in water than \eqref{E:KdV} does;
see \cite{BKN}, for instance.
Moreover Whitham advocated that \eqref{E:whitham1} would explain ``breaking" and ``peaking".
Wave breaking --- bounded solutions with unbounded derivatives --- 
in \eqref{E:whitham1} was analytically verified in \cite{CE-breaking}, for instance, 
while its cusped, periodic traveling wave was numerically supported in \cite{EK2}, for instance.
Benjamin, Bona and Mahony proposed in \cite{BBM} another model 
of water waves in a channel --- the BBM equation --- improving\footnote{
$(\frac{\alpha-\alpha_{BBM}}{\alpha})(\xi)>0.1$ for $|d\xi|> 1.768\dots$.}
the dispersive effects of the KdV equation 
by replacing $\alpha_{KdV}(\xi)$ in \eqref{E:KdV}-\eqref{def:aKdV} by 
\[
\alpha_{BBM}(\xi):=\frac{\sqrt{gd}}{1+\frac16d^2\xi^2}=\sqrt{gd}\Big(1-\frac16d^2\xi^2\Big)+O(d^4\xi^4)
\quad \text{for}\quad d\xi\ll 1.
\] 
But, unfortunately, it fails to capture breaking or peaking. 

\

Benjamin and Feir (see \cite{BF, Benjamin1967}) and, independently, 
Whitham (see \cite{Whitham1967}) formally argued that 
Stokes' periodic waves in water would be unstable, leading to sidebands growth, 
namely the {\em Benjamin-Feir} or {\em modulational} instability, if
\begin{equation}\label{E:BF}
d\k>1.363\dots, 
\end{equation} 
where $\k$ is the wave number, or equivalently, if the Froude number\footnote{
$\alpha(0)=\sqrt{gd}$ is the infinitesimal, long wave speed.}
$\alpha(d\k)/\alpha(0)$ is less than $0.8\dots$. 
A rigorous proof may be found in \cite{BM1995}, albeit in the finite depth case. 
Concerning short and intermediately long waves in a channel of water, 
the Benjamin-Feir instability may not manifest in long wavelength approximations, 
such as the KdV and BBM equations. As a matter of fact, their periodic traveling waves 
are spectrally stable; see \cite{BJ1,BD} and \cite{J3}, for instance.
On the other hand, Whitham's model of surface water waves is relevant to all wavelengths, 
and hence it may capture short waves' instabilities, which is the subject of investigation here. 
In particular, we shall derive a criterion governing spectral instability 
of small-amplitude, periodic traveling waves of \eqref{E:whitham1}.

\begin{theorem}[Modulational instability index]\label{thm:main}
A $2\pi/\k$-periodic traveling wave of \eqref{E:whitham1} with sufficiently small amplitude 
is spectrally unstable to long wavelengths perturbations if $\Gamma(d\k)<0$, where 
\begin{equation}\label{def:Gamma}
\Gamma(z):=3\alpha(z)-2\alpha(2z)-1+z\alpha'(z).
\end{equation}
It is spectrally stable to square integrable perturbations if $\Gamma(d\k)>0$. 
\end{theorem}

The {\em modulational instability index} $\Gamma(z)$ in \eqref{def:Gamma} 
is negative for $z>0$ sufficiently large and 
it is positive for $z>0$ sufficiently small (see \eqref{E:LambdaLimit}). 
Therefore a small-amplitude, $2\pi/\k$-periodic traveling wave of \eqref{E:whitham1} 
is modulationally unstable if $\k>0$ is sufficiently large whereas 
it is spectrally stable to square integrable perturbations if $\k>0$ is sufficiently small.
Moreover $\Gamma(z)$ takes at least one transversal root, 
corresponding to change in stability. 

Unfortunately it seems difficult to analytically study 
the sign of the modulational instability index in \eqref{def:Gamma} further. 
Nevertheless $\Gamma(z)$ is explicit, involving hyperbolic functions
(see \eqref{def:M1}), and hence it is amenable of numerical evaluation. 
The graph in Figure~\ref{F:GammaFig} represents $\Gamma(z)$ for $0<z<1.5$, 
from which it is evident that $\Gamma$ takes a unique root $z_c\approx 1.146$ 
over the interval $(0,1.5)$ and $\Gamma(z)>0$ for $0<z<z_c$. 
The graph in Figure~\ref{F:GammaInvertFig} represents $z\mapsto \Gamma(1.146z^{-1})$
for $0<z<1.5$, from which it is evident that $\Gamma(z)<0$ for $z>z_c$. 
Together we may determine the modulational instability versus spectral stability of 
small-amplitude, periodic traveling waves of \eqref{E:whitham1}.

\begin{figure}[h]
\includegraphics[scale=0.7]{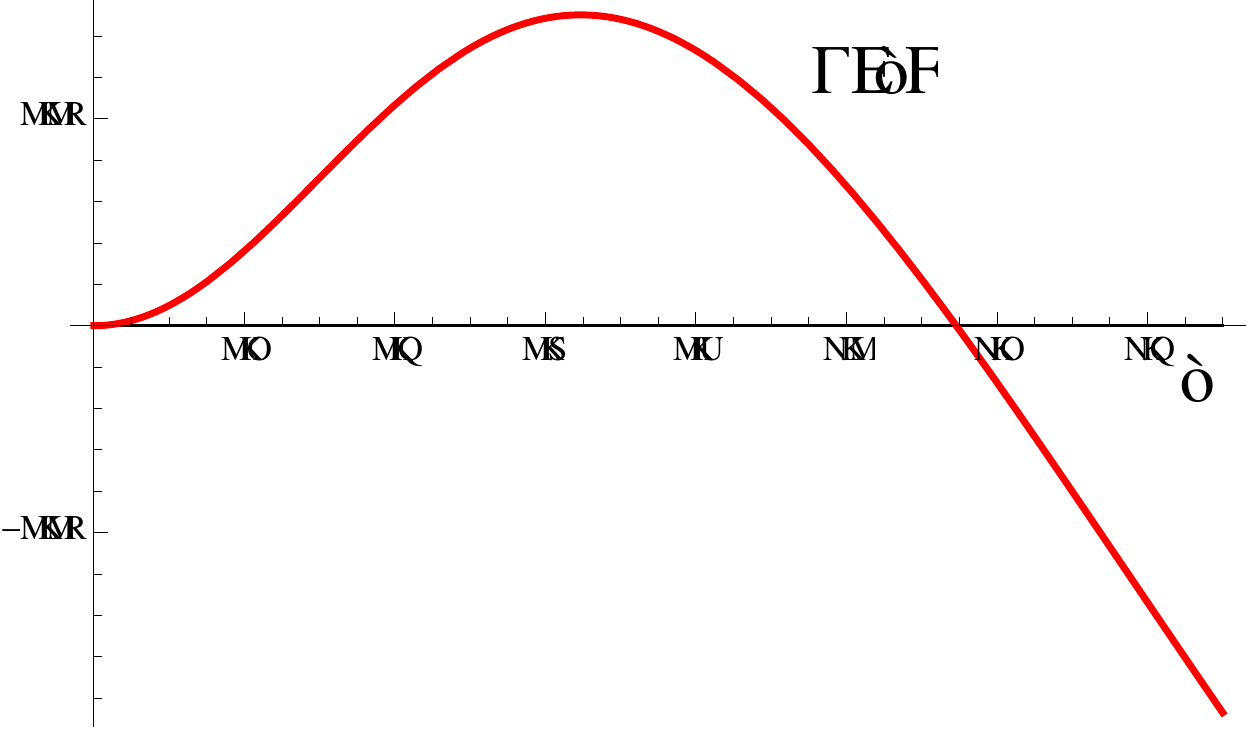}
\caption{The graph of $\Gamma(z)$. 
It is evident that $\Gamma$ takes one root $z_c$ over the interval $(0,1.5)$.
Zooming in upon the root reveals that $z_c\approx 1.146$.}
\label{F:GammaFig}\end{figure}

\begin{figure}[h]
\includegraphics[scale=0.7]{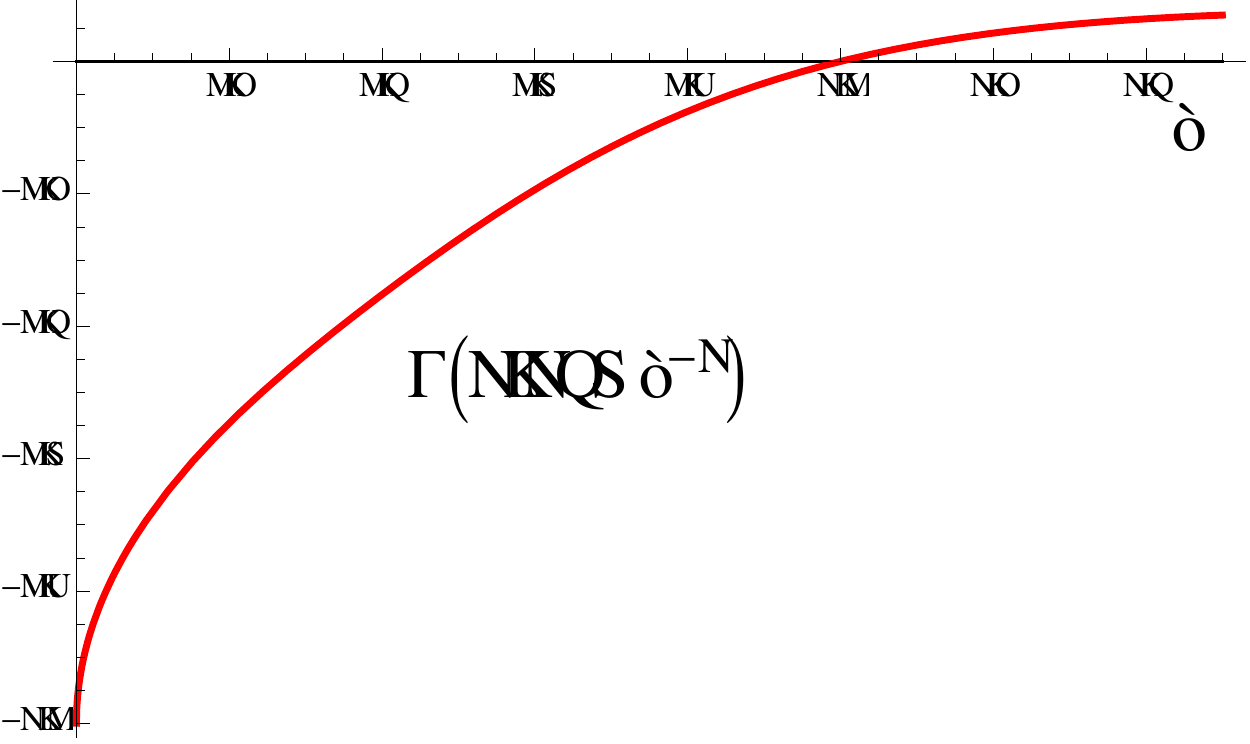}
\caption{The graph of $\Gamma(1.146z^{-1})$. 
It is evident that $\Gamma(z)<0$ for $z>z_c\approx1.146$.}
\label{F:GammaInvertFig}\end{figure}

\begin{corollary}[Modulational instability vs. spectral stability]\label{cor:main}
A $2\pi/\k$-periodic traveling wave of \eqref{E:whitham1} with sufficiently small amplitude
is modulationally unstable if the wave number is supercritical, i.e., if
\begin{equation}\label{E:BF-whitham}
d\k>z_c\approx1.146,
\end{equation}
where $z_c$ is the unique root of $\Gamma(z)$ for $0<z<\infty$. 
It is spectrally stable to square integrable perturbations if $0<d\k<z_c$. 
\end{corollary}

Corollary~\ref{cor:main} qualitatively reproduces the Benjamin-Feir instability (see \eqref{E:BF})
in \cite{BF,Whitham1967} and \cite{BM1995} of Stokes waves in a channel of water;  
the critical Froude number corresponding to $d\k\approx 1.146$ is approximately $0.844\dots$ 
(To compare, the critical Froude number corresponding 
to $d\k\approx 1.363$ is approximately $0.802\dots$). 
We remark that the critical wave number (see \eqref{E:BF}) or the critical Froude number
for the Benjamin-Feir instability was obtained
in \cite{BF, Whitham1967} and \cite{BM1995} through an approximation of 
the numerical value of an appropriate index, called the dispersion equation.
It is therefore not surprising that the proof of Corollary~\ref{cor:main} ultimately relies upon 
a numerical evaluation of the modulational instability index. 

Corollary~\ref{cor:main} furthermore discloses that 
small-amplitude and long-wavelength, periodic traveling waves of \eqref{E:whitham1} 
are spectrally stable. To compare, its small-amplitude, solitary waves 
are conditionally and nonlinearly stable; see \cite{EGW}, for instance. 
Recent numerical studies in \cite{DO}, however, suggest that 
Stokes waves are subject to short wavelength instabilities, regardless of the wave number.  
Perhaps a bi-directional Whitham's model will capture such instabilities,
which is currently under investigation.

\

Recently an extensive body of work has aimed at translating formal modulation theories 
in \cite{BF,Whitham1967}, for instance, into rigorous mathematical results. 
It would be impossible to do justice to all advances in the direction, but we may single out a few --- 
\cite{OZ2003a, OZ2003b,S} for viscous conservation laws, 
\cite{DS} for nonlinear Schr\"odinger equations,
and \cite{BJ1, JZ2} for generalized KdV equations. 
The proofs, unfortunately, rely upon Evans function techniques and other ODE methods, 
and hence they may not be applicable to {\em nonlocal} equations.
Note incidentally that \eqref{E:whitham1} is nonlocal.
Nevertheless, in \cite{BHV} and \cite{J2013}, 
a rigorous, long wavelength perturbation calculation
was carried out for a class of nonlinear nonlocal equations, via a spectral perturbation 
of the associated linearized operator with respect to the Floquet exponent.
Moreover modulational instability was determined 
either with the help of the variational structure (see \cite{BHV})
or using the small amplitude wave asymptotics (see \cite{J2013}).  

Here we take matters further and derive a modulational instability index for \eqref{E:whitham1}.
Specifically we shall employ Fourier analysis and a perturbation argument 
to compute the spectra of the associated, one-parameter family of Bloch operators, 
deducing that its small-amplitude, periodic traveling wave is spectrally stable 
to short wavelengths perturbations 
whereas three eigenvalues near zero may lead to instability to long wavelengths perturbations.
We shall then examine how the spectrum at the origin (for the zero Floquet exponent) 
bifurcates for small Floquet exponents and for small amplitude parameters.


\

The dispersion symbol of \eqref{E:whitham1}, in addition to being nonlocal, 
is {\em inhomogeneous}, i.e., it does not scale like a homogeneous function, 
and hence it seems difficult to determine (in)stability 
without recourse to the small amplitude wave asymptotics. 
On the other hand, it may lose physical relevances for medium and large amplitude waves. 

\

The present development is similar to those in \cite{GH,HK,H08,HLS,H11}, among others, where the authors likewise use spectral
perturbation theory to analyze the spectral stability stability of small amplitude periodic traveling wave solutions to nonlinear dispersive equations
such as the nonlinear Schr\"odinger, generalized KdV, Benjamin-Bona-Mahony, Kawahara, and generalized Kadomtsev-Petviashvili equations.
In the present work, however, we allow for nonlocal and inhomogeneous dispersion symbols.  
In particular, the present development may readily be adapted to 
a broad class of nonlinear dispersive equations of the form \eqref{E:whitham1}.
To illustrate this, we shall derive in Section~\ref{sec:extension} modulational instability indices 
for small-amplitude, periodic traveling waves for a family of 
KdV type equations with fractional dispersion, reproducing the results in \cite{J2013}, 
and for the intermediate long wave equation, discovering spectral stability.

\

While preparing the manuscript, the authors learned (see \cite{SKCK2014}) that 
John Carter and his collaborators numerically computed 
the spectrum of the linearized operator associated with \eqref{E:whitham1} 
and discovered the Benjamin-Feir instability when the wave number of the underlying wave 
times the undisturbed fluid depth is greater than $1.145\dots$, 
which is in good agreement with the analytical finding here (see \eqref{E:BF-whitham}).
In the case of modulational instability, furthermore, they showed that 
the spectrum is in a ``figure 8" configuration, similar to those for generalized KdV equations
(see \cite{HK, BJ1}, for instance). 
An analytical proof is well-beyond perturbative arguments, however. 

\subsection*{Notation} 
Let $L^p_{2\pi}$ in the range $p\in [1,\infty]$ denote the space of 
$2\pi$-periodic, measurable, real or complex valued functions over $\mathbb{R}$ such that 
\[
\|f\|_{L^p_{2\pi}}=\Big(\frac{1}{2\pi}\int^\pi_{-\pi} |f|^p~dx\Big)^{1/p}<+\infty \quad \text{if}\quad p<\infty
\]
and essentially bounded if $p=\infty$. 
Let $H^1_{2\pi}$ denote the space of $L^2_{2\pi}$-functions such that $f' \in L^2_{2\pi}$
and $H^\infty_{2\pi}=\bigcap_{k=0}^\infty H^k_{2\pi}$.

We express $f \in L^1_{2\pi}$ in the Fourier series as
\[
f(z) \sim \sum_{n\in \mathbb{Z}} \widehat{f}_n e^{inz}, \quad\text{where}\quad
\widehat{f}_n=\frac{1}{2\pi}\int^\pi_{-\pi}f(z)e^{-inz}~dz.
\]
If $f \in L^p_{2\pi}$, $p>1$, then the Fourier series converges to $f$ pointwise almost everywhere. 
We define the half-integer Sobolev space $H^{1/2}_{2\pi}$ via the norm 
\[
\|f\|_{H^{1/2}_{2\pi}}^2=\widehat{f}_0^2+\sum_{n\in\mathbb{Z}}|n||\widehat{f}_n|^2.
\]
We define the (normalized) $L^2_{2\pi}$-inner product as
\begin{equation}\label{def:i-product}
\langle f,g\rangle=\frac{1}{2\pi}\int^\pi_{-\pi} f(z)\overline{g}(z)~dz
=\sum_{n\in\mathbb{Z}} \widehat{f}_n\overline{\widehat{g}_n}.
\end{equation}

\

\section{Existence of periodic traveling waves}\label{sec:existence}

We shall discuss how one constructs
small-amplitude, periodic traveling waves of the Whitham equation for surface water waves, 
after normalization of parameters,
\begin{equation}\label{E:whitham}
\partial_t u+\M\partial_x u+\partial_x(u^2)=0,
\end{equation}
where, abusing notation,
\begin{equation}\label{def:M}
\widehat{\M f}(\xi)=\sqrt{\frac{\tanh\xi}{\xi}}\widehat{f}(\xi)=:\alpha(\xi)\widehat{f}(\xi).
\end{equation}
Specifically \eqref{E:whitham} is obtained from \eqref{E:whitham1} via the scaling
\begin{equation}\label{E:normalization}
t\mapsto d\sqrt{gd}\,t, \quad x\mapsto dx\quad\text{and}\quad u\mapsto \frac43du.
\end{equation}

Note that $\alpha$ is real analytic, even and strictly decreasing to zero over the interval $(0,\infty)$. 
Since its inverse Fourier transform lies in $L^1(\mathbb{R})$ 
(see \cite[Theorem~4.1]{EK}, for instance), $\M:L^2(\mathbb{R}) \to L^2(\mathbb{R})$ is bounded. 
Furthermore, since 
\[
|\alpha^{(j)}(\xi)|\leq \frac{C}{(1+|\xi|)^{j+1/2}} \quad 
\text{pointwise in $\mathbb{R}$}\quad \text{for $j=0,1,2,\dots$}
\]
for some constant $C=C(j)>0$, it follows that
$\M:H^s(\mathbb{R}) \to H^{s+1/2}(\mathbb{R})$ is bounded for all $s\geq 0$. 
Loosely speaking, therefore, $\M$ behaves\footnote{
The dispersion symbol of \eqref{E:whitham} induces no smoothing effects, 
and hence its well-posedness proof in spaces of low regularity 
via techniques in nonlinear dispersive equations seems difficult. 
For the present purpose, however, a short-time well-posedness theory suffices.
As a matter of fact, energy bounds and a successive iteration method 
lead to the short time existence in $H^{3/2+}_{2\pi}$.}
like $|\partial_x|^{-1/2}$. Note however that $\alpha$ 
does not scale like a homogeneous function.
Hence \eqref{E:whitham} possesses no scaling invariance.

\

A traveling wave solution of \eqref{E:whitham} takes the form $u(x,t)=u(x-ct)$, 
where $c>0$ and $u$ satisfies by quadrature that 
\begin{equation}\label{E:traveling}
\M u-cu+u^2=(1-c)^2b
\end{equation}
for some $b\in\mathbb{R}$. 
In other words, it propagates at a constant speed $c$ without change of shape. 
We shall seek a solution of \eqref{E:traveling} of the form
\[
u(x)=w(z), \qquad z=\k x,
\]
where $\k>0$, interpreted as the wave number, and $w$ is $2\pi$-periodic, satisfying that
\begin{equation}\label{E:periodic}
\M_\k w-cw+w^2=(1-c)^2b.
\end{equation}
Here and elsewhere, 
\begin{equation}\label{def:Mk}
\M_\k e^{inz}=\alpha(\k n)e^{inz}\quad \text{for}\quad n \in \mathbb{Z},
\end{equation}
or equivalently, $\M_\k(1)=1$ and
\[
\M_\k(\cos nz)=\alpha(\k n)\cos nz\quad \text{and}\quad
\M_\k(\sin nz)=\alpha(\k n)\sin nz\quad \text{for } n\geq 1,
\]
and it is extended by linearity and continuity. Note that 
\begin{equation}\label{E:Mk}
\M_\k:H^s_{2\pi}\to H^{s+1/2}_{2\pi}\quad\text{for all $\k>0$}\quad\text{for all $s\geq 0$} 
\end{equation}
is bounded. 

\

As a preliminary we record the smoothness of solutions of \eqref{E:periodic}.

\begin{lemma}[Regularity]\label{lem:regularity}
If $w \in H^1_{2\pi}$ satisfies \eqref{E:periodic} for some $c>0$, $\k>0$ and $b\in\mathbb{R}$ 
and if $c-2\|w\|_{L^\infty}\geq \epsilon >0$ for some $\epsilon$ then $w\in H^\infty_{2\pi}$.
\end{lemma}

\begin{proof}
The proof is similar to that of \cite[Lemma~2.3]{EGW} in the solitary wave setting. 
Hence we omit the detail. 
\end{proof}


Let
\[ 
F(w;\k,c,b)=\M_\k w-cw+w^2-(1-c)^2b
\]
and note from \eqref{E:Mk} and a Sobolev inequality that 
$F:H^1_{2\pi}\times \mathbb{R}_+\times \mathbb{R}_+\times \mathbb{R} \to H^1_{2\pi}$.  
We shall seek a non-trivial solution $w\in H^1_{2\pi}$ and $\kappa,c>0$, $b\in\mathbb{R}$ of 
\begin{equation}\label{E:F}
F(w;\k,c,b)=0,
\end{equation}
which, by virtue of Lemma~\ref{lem:regularity}, 
provides a non-trivial smooth $2\pi$-periodic solution of \eqref{E:periodic}.  
Note that 
\[
\partial_wF(w;\k,c,b)v=(\M_\k-c+2w)v \in H^1_{2\pi}, \quad v\in H^1_{2\pi},
\] 
and $\partial_\k F(w;\k,c,b)\delta:=\M'_\delta w$, $\delta \in \mathbb{R}$, are continuous, where 
\[
\M'_\delta e^{inz}=\delta n\alpha'(\k n)e^{inz} \quad \text{for}\quad n\in\mathbb{Z}.
\] 
Since 
\[
\partial_cF(w;\k,c,b)=-w+2(1-c)b\quad\text{and}\quad \partial_bF(w;\k,c,b)=-(1-c)^2
\] 
are continuous we deduce that 
$F:H^1_{2\pi} \times \mathbb{R}_+\times \mathbb{R}_+\times \mathbb{R} \to H^1_{2\pi}$ is $C^1$. 
Furthermore, since Fr\'echet derivatives of $F$ with respect to $w$ and $c,b$ 
of all orders greater than three are zero everywhere and since $\alpha$ is a real-analytic function,  
we conclude that $F$ is a real-analytic operator.

\

Observe that\footnote{Moreover $F(c;\k,c,0)=0$ for all $\k,c>0$, 
but we discard them in the interest of small amplitude solutions.}
\[
F(0;\k,c,0)=0
\]
for all $\k,c>0$ and 
\[
\ker(\partial_wF(w_0;\k,c,0))=\ker(\M_\k-c)=\text{span}\{e^{\pm iz}\}
\]
provided that $c=\alpha(\k)$; the kernel is trivial otherwise.  
In the case of $b=0$, a one-parameter family of solutions of \eqref{E:F} in $H^1_{2\pi}$,  
and hence smooth $2\pi$-periodic solutions of \eqref{E:periodic}, was obtained in \cite{EK, EK2} 
via a local bifurcation theorem near the zero solution and $c=\alpha(\k)$
for each\footnote{The proof in \cite{EK, EK2} is merely for $\k=1$, 
but the necessary modifications to consider general $\k>0$ are trivial.} $\k>0$.
Moreover their small amplitude asymptotics was calculated. 
For $|b|$ sufficiently small, we then appeal to the Galilean invariance of \eqref{E:periodic}, under
\begin{equation}\label{E:galilean}
w(z)\mapsto w(z)+v \quad\text{and}\quad c\mapsto c-2v,\quad
(1-c)^2b\mapsto (1-c)^2b+(1-c)v+v^2
\end{equation}
for any $v\in\mathbb{R}$, to construct small-amplitude solutions of \eqref{E:periodic} for $|b|\ll 1$.  This analysis is
summarized in the following.


\begin{proposition}[Existence]\label{prop:existence}
For each $\k>0$ and for each $|b|$ sufficiently small, a family of periodic traveling waves of \eqref{E:whitham} exists and 
\[
u(x,t)=w(a,b)(\k (x-c(\k, a,b)t))=:w(\k, a, b)(z)
\] 
for $|a|$ sufficiently small, where $w$ and $c$ depend analytically upon $\k$, $a$, $b$. 
Moreover $w$ is smooth, even and $2\pi$-periodic in $z$, and $c$ is even in $a$. Furthermore, 
\begin{align}\label{E:w(k,a,b)}
w(\k, a,b)(z)=w_0(\k,b)&+a\cos z\\ &+
\frac12a^2\Big(\frac{1}{\alpha(\k)-1}+\frac{\cos(2z)}{\alpha(\k)-\alpha(2\k)}\Big)+O(a(a^2+b^2))\notag
\intertext{and} 
c(\k,a,b)=c_0(\k,b)& +a^2\Big(\frac{1}{\alpha(\k)-1}+\frac12\frac{1}{\alpha(\k)-\alpha(2\k)}\Big)
+O(a(a^2+b^2))\label{E:c(k,a,b)}
\end{align}
as $|a|, |b| \to 0$, where 
\begin{align*}
c_0(\k,b):=&\frac{\alpha(\k)+2b-2b^2+O(b^3)}{1+2b-2b^2+O(b^3)} \\
=&\alpha(\k)+2b(1-\alpha(\k))-6b^2(1-\alpha(\k))+O(b^3)
\intertext{and} 
w_0(\k,b):=&b(1-\alpha(\k))-b^2(1-\alpha(\k))+O(b^3).
\end{align*}
\end{proposition}

\section{Modulational instability vs. spectral stability}\label{sec:stability}

Throughout the section, let $w=w(\k,a,b)$ and $c=c(\k,a,b)$, for $\k>0$ and $|a|,|b|$ sufficiently small,
form a small-amplitude, $2\pi/\k$-periodic traveling wave of \eqref{E:whitham}, 
whose existence follows from the previous section. 
We shall address its modulational instability versus spectral stability.

\

Linearizing \eqref{E:whitham} about $w$ in the frame of reference moving at the speed $c$, 
we arrive, after recalling $z=\k x$, at that 
\[
\partial_t v+\k\partial_z(\M_\k-c+2w)v=0.
\]
Seeking a solution of the form $v(z,t)=e^{\mu\k t}v(z)$, 
$\mu\in\mathbb{C}$ and $v\in L^2(\mathbb{R})$, moreover, 
we arrive at the pseudo-differential spectral problem 
\begin{equation}\label{E:eigen} 
\mu v=\partial_z(-\M_\k+c-2w)v=:\L(\k,a,b) v.
\end{equation}
We then say that $w$ is {\em spectrally unstable} if the $L^2(\mathbb{R})$-spectrum 
of $\L$ intersects the open, right half plane of $\mathbb{C}$
and it is (spectrally) {\em stable} otherwise. 
Note that $v$ is an arbitrary, square integrable perturbation.
Since \eqref{E:eigen} remains invariant under 
\[
v\mapsto \bar{v}\quad\text{and}\quad\mu\mapsto \bar{\mu}
\] 
and under
\[
z\mapsto -z\quad\text{and}\quad \mu\mapsto -\mu,
\]
the spectrum of $\L$ is symmetric with respect to the reflections about the real and imaginary axes. 
Consequently $w$ is spectrally unstable if the $L^2(\mathbb{R})$-spectrum of $\L$ 
is {\em not} contained in the imaginary axis. 

\

In the case of (local) generalized KdV equations, 
for which the nonlinearity in \eqref{E:KdV} is arbitrary, 
the $L^2(\mathbb{R})$-spectrum of the associated linearized operator,
which incidentally is purely essential, was related in \cite{BJ1}, for instance, 
to eigenvalues of the monodromy map via the periodic Evans function.
Furthermore modulational instability was determined in terms of conserved quantities of the PDE 
and their derivatives with respect to constants of integration arising in the traveling wave ODE.
Confronted with {\em nonlocal} operators, however, 
Evans function techniques and other ODE methods may not be applicable. 
Following \cite{BHV} and \cite{J2013}, instead, 
we utilize Floquet theory and make a Bloch operator decomposition. 
As a matter of fact (see \cite{RSIV} and \cite[Proposition~3.1]{J2013}, for instance), 
$\mu\in\mathbb{C}$ is in the~$L^2(\mathbb{R})$-spectrum
of $\mathcal{L}$ if and only if there exists a $2\pi$-periodic function $\phi$ and a $\tau\in[-1/2,1/2)$, known as the {\em Floquet exponent}, such that
\begin{equation}\label{E:Ltau}
\mu\phi=e^{-i\tau z}\L(\k,a,b)e^{i\tau z}\phi=:\L_\tau(\k,a,b)\phi.
\end{equation}
Consequently 
\[
\text{spec}_{L^2(\mathbb{R})}(\L(\k,a,b))=
\bigcup_{\tau \in [-1/2,1/2)} \text{spec}_{L^2_{2\pi}}(\L_\tau(\k,a,b)).
\]
Note that the $L^2_{2\pi}$-spectrum of $\L_\tau$ consists merely of discrete eigenvalues
for each~$\tau \in [-1/2,1/2)$. Thereby we parametrize 
the essential $L^2(\mathbb{R})$-spectrum of $\L$ 
by the one-parameter family of point $L^2_{2\pi}$-spectra 
of the associated {\em Bloch operators} $\L_\tau$'s. Since 
\begin{equation}\label{E:sym}
\text{spec}_{L^2_{2\pi}}(\L_\tau)=\overline{\text{spec}_{L^2_{2\pi}}(\L_{-\tau})},
\end{equation}
incidentally, it suffices to take $\tau\in [0,1/2]$.

\subsection*{Notation}In what follows, $\kappa>0$ is fixed and 
suppressed to simplify the exposition, unless specified otherwise. 
Thanks to Galilean invariance (see \eqref{E:galilean}) we may take that $b=0$. We write 
\[
\L_{\tau,a}=\L_\tau(\k,a,0).
\]
and we use $\sigma$ for the $L^2_{2\pi}$-spectrum. 

\

Note in passing that $\tau=0$ indicates periodic perturbations 
with the same period as the underlying waveform
and $|\tau|$ small physically amounts to 
long wavelengths perturbations or slow modulations of the underlying wave.

\subsection{Spectra of the Bloch operators}\label{sec:spec}
We set forth the study of the $L^2_{2\pi}$-spectra of $\L_{\tau,a}$'s 
for $\tau\in[0,1/2]$ and $|a|$ sufficiently small.

\

We begin by discussing 
\[
\L_{\tau,0}=(\partial_z+i\tau)(-e^{-i\tau z}\M_\k e^{i\tau z}+\alpha(\k))
\]
for $\tau \in [0,1/2]$, 
corresponding to the linearization of \eqref{E:whitham} about the zero solution and $c=\alpha(\k)$
(see \eqref{E:w(k,a,b)} and \eqref{E:c(k,a,b)}). 
A straightforward calculation reveals that 
\begin{equation}\label{E:L(a=0)}
\L_{\tau,0}e^{inz}=i\omega_{n,\tau}e^{inz}
\quad\text{for all $n \in \mathbb{Z}$\quad\text{for all $\tau \in [0,1/2]$},}
\end{equation}
where
\begin{equation}\label{def:omega}
\omega_{n,\tau}=(n+\tau)(\alpha(\k)-\alpha(\k n+\k\tau)).
\end{equation}
Consequently
\[
\sigma(\L_{\tau,0})=\{i\omega_{n,\tau}:n\in \mathbb{Z}\} \subset i\mathbb{R}
\] 
for all $\tau \in [0,1/2]$, and hence the zero solution of \eqref{E:whitham} 
is spectrally stable to square integrable perturbations. 

In the case of $\tau=0$, clearly,
\[
\omega_{-1,0}=\omega_{0,0}=\omega_{1,0}=0.
\]
Since $\omega_{n,0}$ is symmetric and strictly increasing for $|n|\geq 2$, moreover, 
\[
\cdots<\omega_{-3,0}<\omega_{-2,0}<0<\omega_{2,0}<\omega_{3,0}<\cdots.
\]
In particular, zero is an eigenvalue of $\L_{0,0}$ with algebraic multiplicity three.
In the case of $\tau \in (0,1/2]$, since $\omega_{n,\tau}$ is decreasing for $-1/2<n+\tau<1/2$ 
and increasing for $n+\tau<-1$ or $n+\tau>1$, similarly,
\[
\omega_{0,1/2}\leq \omega_{0,\tau}, \omega_{\pm 1, \tau} \leq \omega_{1,1/2}
\]
and
\[
\cdots<\omega_{-3,\tau}<\omega_{-2,\tau}<\omega_{0,1/2}
<\omega_{1,1/2}<\omega_{2,\tau}<\omega_{3,\tau}<\cdots.
\]
Therefore we may write that
\begin{align}\label{E:sigma12}
\sigma(\L_{\tau,0})=&\{i\omega_{-1,\tau}, i\omega_{0,\tau}, i\omega_{1,\tau}\}
\bigcup \{i\omega_{n,\tau}: |n|\geq 2\} \notag \\
=:&\sigma_1(\L_{\tau,0}) \bigcup \sigma_2(\L_{\tau,0})
\end{align}
for each $\tau\in[0,1/2]$, where $\sigma_1(\L_{\tau,0})\bigcap \sigma_2(\L_{\tau,0})=\emptyset$. 
Note that
\[ 
\L_{\tau,0}=e^{-i\tau z}\partial_z(-\M_\k+\alpha(\k))e^{i\tau z}
=(\partial_z+i\tau)e^{-i\tau z}(-\M_\k+\alpha(\k))e^{i\tau z}.
\]
Note moreover that if $\phi$ lies in the (generalized) $L^2_{2\pi}$-spectral subspace 
associated with $\sigma_2(\L_{\tau,0})$, 
i.e., if $\phi(z)\sim \sum_{|n|\geq 2} \widehat{\phi}(n)e^{inz}$, then 
\begin{align*}
\langle(e^{-i\tau z}(-\M_\k+\alpha(\k))e^{i\tau z}\phi, \phi \rangle
=& \langle (\alpha(\k)-\M_\k)e^{i\tau z}\phi, e^{i\tau z}\phi \rangle \\
\geq& (\alpha(\k)-\alpha(3\k/2))\|\phi\|_{L^2}^2 \notag
\end{align*}
uniformly for $\tau \in[0,1/2]$.  Consequently $\sigma_2(\L_{\tau,0})$ consists of 
infinitely many, simple and purely imaginary eigenvalues for all $\tau\in[0,1/2]$,
and the bilinear form induced by $\mathcal{L}_{\tau,0}$
is strictly positive definite on the associated total eigenspace uniformly for $\tau\in[0,1/2]$.

\

To proceed, for $|a|$ sufficiently small, since 
\[
\|\L_{\tau,a}-\L_{\tau,0}\|_{L^2_{2\pi} \to L^2_{2\pi}}=O(|a|)
\quad\text{as $|a| \to 0$}\quad\text{uniformly for $\tau \in [0,1/2]$}
\] 
by brutal force, a perturbation argument implies that
$\sigma(\L_{\tau,a})$ is close to $\sigma(\L_{\tau,0})$. 
Therefore we may write that 
\[
\sigma(\L_{\tau,a})=\sigma_1(\L_{\tau,a})\bigcup \sigma_2(\L_{\tau,a}),\quad\sigma_1(\L_{\tau,a})\bigcap \sigma_2(\L_{\tau,a})=\emptyset
\]
for each $\tau \in[0,1/2]$ and $|a|$ sufficiently small;
$\sigma_1(\L_{\tau,a})$ contains three eigenvalues 
near $i\omega_{0,\tau}$, $i\omega_{\pm1, \tau}$, 
which depend continuously upon $a$ for $|a|$ sufficiently small, 
and $\sigma_2(\L_{\tau,a})$ is made up of infinitely many simple eigenvalues 
and the bilinear form induced by $\L_{\tau,a}$ is strictly positive definite 
on the associated eigenspaces 
uniformly for $\tau\in[0,1/2]$ and $|a|$ sufficiently small. In particular, 
if $\phi$ lies in the (generalized) $L^2_{2\pi}$-spectral subspace 
associated with $\sigma_2(\L_{\tau,a})$ then
\[
\left<\phi,\L_{\tau,a}\phi\right>\geq C\|\phi\|_{L^2}^2
\]
for some $C>0$, for all $\tau\in[0,1/2]$ and $|a|$ sufficiently small.

Since $\sigma(\L_{\tau,a})$ is symmetric about the imaginary axis, 
eigenvalues of $\L_{\tau,a}$ may leave the imaginary axis, leading to instability, as $\tau$ and $a$ vary
only through collisions with other purely imaginary eigenvalues.  
Moreover if a pair of non-zero eigenvalues of $\L_{\tau,a}$ collide 
on the imaginary axis then they will either remain on the imaginary axis 
or they will bifurcate in a symmetric pair off the imaginary axis, 
i.e, they undergo a Hamiltonian Hopf bifurcation.  
A necessary condition for a Hamiltonian Hopf bifurcation is that 
the colliding, purely imaginary eigenvalues have {\em opposite Krein signature}; 
see \cite[Proposition~7.1.14]{KP13}, for instance. 
We say that an eigenvalue $\mu \in \mathbb{R}i\setminus\{0\}$ of $\L_{\tau,a}$ has negative Krein signature
if the matrix $(\left<\phi_j^\mu,\L_{\tau,a}\phi_k^\mu\right>)_{j,k}$ 
restricted to the finite-dimensional (generalized) eigenspace  has at least one negative eigenvalue, 
where $\phi_j$'s are the corresponding (generalized) eigenfunctions,
and it has positive Krein signature if all eigenvalues of the matrix are positive. 
In particular, a simple eigenvalue $\mu\in\mathbb{R}i\setminus\{0\}$ has negative Krein signature
if $\left<\phi,\L_{\tau,a}\phi\right><0$, where $\phi$ is the corresponding eigenfunction,
and it has positive Krein signature if $\left<\phi,\L_{\tau,a}\phi\right>>0$;
see \cite[Section~7.1]{KP13}, for instance, for the detail. 
For all $\tau\in[0,1/2]$ and $|a|$ sufficiently small, therefore, 
all eigenvalues in $\sigma_2(\L_{\tau,a})$ have positive Krein signature,
and hence they \emph{are restricted to the imaginary axis}, 
contributing to spectral stability for all $\tau\in[0,1/2]$ and $|a|$ sufficiently small.

\

It remains to understand the location of three eigenvalues 
associated to the spectral subspace $\sigma_1(\L_{\tau,a})$ 
for $\tau\in[0,1/2]$ and $|a|$ sufficiently small.  Since
\[
|\omega_{j,\tau}-\omega_{k,\tau}|\geq \omega_0>0 
\quad \text{for $j,k=-1,0,1$}\quad\text{and $j\neq k$,}  
\]
for $\tau\geq \tau_0>0$ for any $\tau_0$ and for some $\omega_0$, by brutal force,
eigenvalues in $\sigma_1(J_\tau\L_{\tau,a})$ remain simple and distinct 
under perturbations of  $a$  for $\tau\geq\tau_0>0$ for any $\tau_0$ and for $|a|$ sufficiently small. 
In particular,
\[
\sigma_1(\L_{\tau,a})\subset i\mathbb{R}\quad \text{for}\quad 0<\tau_0\leq \tau\leq 1/2
\]
for any $\tau_0$ and for $|a|$ sufficiently small. 
Thereby we conclude that a small-amplitude, periodic traveling wave of \eqref{E:whitham}
is spectrally stable to short wavelengths perturbations.
On the other hand, we establish that three eigenvalues in $\sigma_1(\L_{\tau,a})$ 
collide at the origin for $\tau=0$ and for $|a|$ sufficiently small, possibly leading to instability.

\begin{lemma}[Eigenvalues in $\sigma_1(\L_{0,a})$]\label{L:tau=0}
If $w=w(\k,a,0)$ and $c=c(\k,a,0)$ satisfy \eqref{E:periodic} 
for some $\k>0$, $b=0$ and for $|a|$ sufficiently small 
then zero is a generalized $L^2_{2\pi}$-eigenvalue of $\L_{0,a}=\partial_z(-\M_\k+c-2w)$ with algebraic multiplicity three and geometric multiplicity two. Moreover
\begin{align}
\phi_1(z):=&\frac{1}{2(1-\alpha(\k))}(\partial_bc)(\partial_aw)-(\partial_ac)(\partial_bw))(\k,a,0)(z) \notag \\
=&\cos z+\frac{-1/2+\cos(2z)}{\alpha(\k)-\alpha(2\k)}a+O(a^2), \label{E:phi1} \\
\phi_2(z):=&-\frac{1}{a}(\partial_zw)(\k,a,0)(z) \notag \\
=&\sin z+\frac{\sin(2z)}{\alpha(\k)-\alpha(2\k)}a+O(a^2), \label{E:phi2} \\
\phi_3(z):=&1 \label{E:phi3}
\end{align}
form a basis of the generalized eigenspace of $\L_{0,a}$ for $|a|\ll 1$. Specifically,
\[
\L_{0,a}\phi_1=\L_{0,a}\phi_2=0\quad\text{and}\quad \L_{0,a}\phi_3=-2a\phi_2.
\]
\end{lemma}

For $a=0$, in particular, $i\omega_{-1,0}=i\omega_{0,0}=i\omega_{1,0}=0$.
Note that $\phi_1, \phi_3$ are even functions and $\phi_2$ is odd.

\begin{proof}
Differentiating \eqref{E:periodic} with respect to $z$ and evaluating at $b=0$, we obtain that 
\[
\L_{0,a}(\partial_z w)=0.
\]
Therefore zero is an $L^2_{2\pi}$ eigenvalue of $\L_{a,0}$. 
Incidentally this is reminiscent of that \eqref{E:whitham} remains invariant under spatial translations. 
Differentiating \eqref{E:periodic} with respect to $a$ and $b$, respectively, and evaluating at $b=0$, 
we obtain that 
\[
\L_{0,a}(\partial_a w)=(\partial_a c)(\partial_z w), \quad\text{and}\quad
\L_{0,a}(\partial_b w)=(\partial_b c)(\partial_z w).
\]
Consequently
\[
\L_{0,a}((\partial_bc)(\partial_aw)-(\partial_ac)(\partial_bw))=0,
\]
furnishing a second eigenfunction of $\L_{0,a}$ corresponding to zero. 
A straightforward calculation moreover reveals that 
\[
\L_{0,a}1=2\partial_zw.
\]
After normalization of constants, this completes the proof.
\end{proof}

To decide the modulational instability versus spectral stability 
of small-amplitude, periodic traveling waves of \eqref{E:whitham},
therefore, it remains to understand how the triple eigenvalue at the origin of $\L_{0,a}$
bifurcates for $\tau$ and $|a|$ small.

\subsection{Operator actions and modulational instability}\label{sec:Bta}

We shall examine the $L^2_{2\pi}$-spectra of $\L_{\tau,a}$'s 
in the vicinity of the origin for $\tau$ and $|a|$ sufficiently small,  
and determine the modulational instability versus spectral stability
of small-amplitude, periodic traveling waves of \eqref{E:whitham}. 
Specifically we shall calculate
\begin{equation}\label{def:B}
\mathbf{B}_{\tau,a}=\left( \frac{\l \phi_j, \L_{\tau,a}\phi_k\r}{\l \phi_j, \phi_j\r}\right)_{j,k=1,2,3}
\end{equation}
and 
\begin{equation}\label{def:I}
\mathbf{I}_{a}=\left( \frac{\l \phi_j, \phi_k\r}{\l \phi_j, \phi_j\r}\right)_{j,k=1,2,3}
\end{equation}
up to the quadratic order in $\tau$ and the linear order in $a$ as $\tau, |a| \to 0$, 
where $\phi_1, \phi_2, \phi_3$ are defined in \eqref{E:phi1}-\eqref{E:phi3} 
and form a basis of the spectral subspace associated with $\sigma_1(\L_{0,a})$,
and hence also form a ($\tau$-independent) basis of the spectral subspace associated with1 $\sigma_1(\L_{\tau,a})$ for $\tau$ sufficiently small.
Recall that $\l\,,\r$ denotes the (normalized) $L^2_{2\pi}$-inner product (see \eqref{def:i-product}). 

Note that $\mathbf{B}_{\tau,a}$ and $\mathbf{I}_{a}$ represent, respectively, 
the actions of $\L_{\tau,a}$ and the identity 
on the spectral subspace associated to $\sigma_1(\L_{\tau,a})$. 
In particular, the three eigenvalues in $\sigma_1(\L_{\tau,a})$ agree in location and multiplicity
with the roots of the characteristic equation $\det(\mathbf{B}_{\tau,a}-\mu\mathbf{I}_{a})=0$;
see \cite[Section~4.3.5]{K}, for instance, for the detail.
Consequently the underlying, periodic traveling wave is modulationally unstable 
if the cubic characteristic polynomial $\det(\mathbf{B}_{\tau,a}-\mu\mathbf{I}_{a})$
admits a root with non-zero real part and it is spectrally stable otherwise.

\

For $|a|$ sufficiently small, a Baker-Campbell-Hausdorff expansion reveals that  
\begin{equation}\label{E:Lta}
\L_{\tau,a}=L_{0,a}+i\tau L_{1,a}-\frac12\tau^2L_{2,a}+O(\tau^3)
\end{equation}
as $\tau \to 0$, where 
\begin{align}
L_{0,a}=\L_{0,a}=\;&\partial_z(-\M_\k+\alpha(\k))-2a\partial_z(\cos z)+O(a^2) \notag \\
=:&M_0-2a\partial_z(\cos z)+O(a^2), \label{def:L0} \\
L_{1,a}=[L_{0,a},z]=\;&[-\partial_z\M_k,z]+\alpha(\k)-2a\cos z+O(a^2) \notag \\
=:&M_1-2a\cos z+O(a^2), \label{def:L1} \\
L_{2,a}=[L_{1,a},z]=\;&[[-\partial_z\M_k,z],z]+O(a^2)=:M_2+O(a^2) \label{def:L2}
\end{align}
as $|a| \to 0$. 
Note that $L_{1,a}$ and $L_{2,a}$ are well-defined in $L^2_{2\pi}$ even though $z$ is not. 
For instance,
\[
[\M_\k,z]\partial_z e^{inz}=ne^{inz}\sum_{m\neq 0} \frac{(-1)^{|m|}}{m}
(-\alpha(\k n)-\alpha(\k n+\k m))e^{imz} \quad \text{for }n\in \mathbb{Z}
\]
and $M_1=-[\M_k,z]\partial_z-\M_k$.

\

The Fourier series representations of $M_1$ and $M_2$ seem unwieldy. 
We instead use $\L_{\tau,0}e^{inz}$, $|n|\leq 2$, (see \eqref{E:L(a=0)} and \eqref{def:omega})
to calculate $M_j\phi_k$, $j=0,1,2$ and $k=1,2,3$, and hence $\L_{\tau,a}\phi_k$, $k=1,2,3$.

\

A Taylor expansion of $\omega_{0,\tau}$ reveals that 
\[
\L_{\tau,0}1=i\omega_{0,\tau}=i\tau(\alpha(\k)-1)+O(\tau^3)
\]
as $\tau\to 0$, whence
\begin{equation}\label{E:M1}
M_0(1)=0, \qquad M_1(1)=\alpha(\k)-1,\qquad M_2(1)=0.
\end{equation}
Similarly,
\begin{align*}
\L_{\tau,0}(\cos z)=&\frac{i}{2}(\omega_{1,\tau}e^{iz}+\omega_{-1,\tau}e^{-iz}) \\
=&-i\tau\k\alpha'(\k)\cos z+\frac12\tau^2(2\k\alpha'(\k)+\k^2\alpha''(\k))\sin z+O(\tau^3)
\end{align*}
as $\tau \to 0$, whence 
\begin{equation}\label{E:Mcosz}
\begin{aligned}
M_0(\cos z)=&0,\\ M_1(\cos z)=&-\k\alpha'(\k)\cos z,\\
M_2(\cos z)=&-(2\k\alpha'(\k)+\k^2\alpha''(\k))\sin z.
\end{aligned}
\end{equation}
Moreover $\L_{\tau,0}(\sin z)=\frac12(\omega_{1,\tau}e^{iz}-\omega_{-1,\tau}e^{-iz})$ and 
\begin{equation}\label{E:Msinz}
\begin{aligned}
M_0(\sin z)=&0,\\ M_1(\sin z)=&-\k\alpha'(\k)\sin z,\\
M_2(\sin z)=&(2\k\alpha'(\k)+\k^2\alpha''(\k))\cos z.
\end{aligned}\end{equation}
Continuing, 
\begin{align*}
\L_{\tau,0}(\cos2z)=&\frac{i}{2}(\omega_{2,\tau}e^{2iz}+\omega_{-2,\tau}e^{-2iz}) \\
=&-2(\alpha(\k)-\alpha(2\k))\sin2z
+i\tau(\alpha(\k)-\alpha(2\k)-2\k\alpha'(\k))\cos2z \\
&+\tau^2(\k\alpha'(2\k)+\tau^2\alpha''(2\k))\sin2z+O(\tau^3)
\end{align*}
as $\tau \to 0$, whence
\begin{equation}\label{E:Mcos2z}
\begin{aligned}
M_0(\cos2z)=&-2(\alpha(\k)-\alpha(2\k))\sin2z, \\
M_1(\cos2z)=&(\alpha(\k)-\alpha(2\k)-2\k\alpha'(\k))\cos2z, \\
M_2(\cos2z)=&-2(\k\alpha'(2\k)+\tau^2\alpha''(2\k))\sin2z.
\end{aligned}\end{equation}
Moreover $\L_{\tau,0}(\sin2z)=\frac{1}{2}(\omega_{2,\tau}e^{2iz}-\omega_{-2,\tau}e^{-2iz})$ and
\begin{equation}\label{E:Msin2z}
\begin{aligned}
M_0(\sin2z)=&2(\alpha(\k)-\alpha(2\k))\cos2z, \\
M_1(\sin2z)=&(\alpha(\k)-\alpha(2\k)-2\k\alpha'(\k))\sin2z, \\
M_2(\sin2z)=&2(\k\alpha'(2\k)+\tau^2\alpha''(2\k))\cos2z.
\end{aligned}\end{equation}

\

To proceed, we substitute \eqref{E:M1}-\eqref{E:Msin2z} into \eqref{E:Lta}-\eqref{def:L2} to obtain that
\begin{align*}
\L_{\tau,a}(1)=&2a\sin z+i\tau(\alpha(\k)-1)-2i\tau a\cos z+O(\tau^3+a^2),\\
\L_{\tau,a}(\cos z)=&2a\sin2z-i\tau\k\alpha'(\k)\cos z-i\tau a(1+\cos2z) \\
&+\tau^2(2\k\alpha'(\k)+\k^2\alpha''(\k))\sin z+O(\tau^3+a^2), \\
\L_{\tau,a}(\sin z)=&-2a\cos2z-i\tau\k\alpha'(\k)\sin z-i\tau a\sin2z \\
&-\tau^2(2\k\alpha'(\k)+\k^2\alpha''(\k))\cos z+O(\tau^3+a^2)
\intertext{as $\tau, |a| \to 0$. Moreover}
\L_{\tau,a}(\cos2z)=&-2(\alpha(\k)-\alpha(2\k))\sin2z+a(\sin z+3\sin3z) \\
&+i\tau(\alpha(\k)-\alpha(2\k)-2\k\alpha'(\k))\cos2z-i\tau a(\cos z+\cos3z) \\
&+\tau^2(\k\alpha'(2\k)+\tau^2\alpha''(2\k))\sin2z+O(\tau^3+a^2)
\intertext{and}
\L_{\tau,a}(\sin2z)=&2(\alpha(\k)-\alpha(2\k))\cos2z-a(\cos z+3\cos3z) \\
&+i\tau(\alpha(\k)-\alpha(2\k)-2\k\alpha'(\k))\sin2z-i\tau a(\sin z+\sin3z) \\
&-\tau^2(\k\alpha'(2\k)+\tau^2\alpha''(2\k))\cos2z+O(\tau^3+a^2)
\end{align*}
as $\tau, |a| \to 0$.

We therefore use \eqref{E:phi1}-\eqref{E:phi3} to calculate that 
\begin{align*}
\langle \phi_1, \L_{\tau,a}\phi_1\rangle=&
\Big\langle \cos z+a\frac{-1/2+\cos2z}{\alpha(\k)-\alpha(2\k)}, 
\L_{\tau,a}\Big(\cos z+a\frac{-1/2+\cos2z}{\alpha(\k)-\alpha(2\k)}\Big)\Big\rangle \\
=&-\frac{i}{2}\tau\k\alpha'(\k)+O(\tau^3+a^2)\\
\intertext{and}
\langle \phi_1, \L_{\tau,a}\phi_2\rangle=&-\langle \phi_2, \L_{\tau,a}\phi_1\rangle
=-\frac12\tau^2\Big(\k\alpha'(\k)+\frac12\k^2\alpha''(\k)\Big)+O(\tau^3+a^2),\\
\langle \phi_1, \L_{\tau,a}\phi_3\rangle=&-\langle \phi_3, \L_{\tau,a}\phi_1\rangle
=-\frac{i}{2}\tau a\Big(2+\frac{\alpha(\k)-1}{\alpha(\k)-\alpha(2\k)}\Big)+O(\tau^3+a^2)
\end{align*}
as $\tau, |a|\to 0$. Moreover
\begin{align*}
\langle \phi_2, \L_{\tau,a}\phi_2\rangle=&-\frac{i}{2}\tau\k\alpha'(\k)+O(\tau^3+a^2), \\ 
\langle \phi_2, \L_{\tau,a}\phi_3\rangle=&a+O(\tau^3+a^2), \\
\langle \phi_3, \L_{\tau,a}\phi_2\rangle=&O(\tau^3+a^2),\\
\langle \phi_3, \L_{\tau,a}\phi_3\rangle=&i\tau(\alpha(\k)-1)+O(\tau^3+a^2)
\end{align*}
as $\tau, |a| \to 0$. 

\

A straightforward calculation, furthermore, reveals that
\[
\langle \phi_1,\phi_1\rangle=\langle \phi_2,\phi_2\rangle=\frac12+O(a^2), 
\qquad \langle \phi_3,\phi_3\rangle=1,
\]
and
\[
\langle \phi_1,\phi_3\rangle=-\frac12\frac{a}{\alpha(\k)-\alpha(2\k)}+O(a^2)
\]
as $|a| \to 0$; $\langle\phi_j, \phi_k\rangle=0$ otherwise by a parity argument.

\

To recapitulate (see \eqref{def:B} and \eqref{def:I}),
\begin{equation}\label{E:B}
\begin{aligned}
\mathbf{B}_{\tau,a}=&a\left(\begin{matrix} 0 & 0 & 0 \\ 0& 0 & 2 \\ 0 & 0 & 0 \end{matrix}\right)
+i\tau\left(\begin{matrix} 
-\k\alpha'(\k) & 0 & 0 \\ 0& -\k\alpha'(\k) & 0 \\ 0 & 0 & \alpha(\k)-1 \end{matrix}\right) \\
&-i\tau a\Big(1+\frac12\frac{\alpha(\k)-1}{\alpha(\k)-\alpha(2\k)}\Big)
\left(\begin{matrix} 0 & 0 & 2 \\ 0& 0 & 0 \\ 1 & 0 & 0 \end{matrix}\right) \\
&+\tau^2\Big(\k\alpha'(\k)+\frac12\k^2\alpha''(\k)\Big)
\left(\begin{matrix} 0 & -1 & 0 \\ 1 & 0 & 0 \\ 0 & 0 & 0 \end{matrix}\right)+O(\tau^3+a^2)
\end{aligned}\end{equation}
and 
\begin{equation}\label{E:I}
\mathbf{I}_{a}=\mathbf{I}-\frac{a}{\alpha(\k)-\alpha(2\k)}
\left(\begin{matrix} 0 & 0 & 1 \\ 0& 0 & 0 \\ 1/2 & 0 & 0 \end{matrix}\right)+O(a^2)
\end{equation}
as $|a|\to 0$. Note that 
\[
\mathbf{B}_{0,a}=a\left(\begin{matrix} 0 & 0 & 0 \\ 0 & 0 & 2 \\ 0 & 0 & 0\end{matrix}\right)+O(a^2)
\]
as $|a| \to 0$. Therefore the underlying, periodic traveling wave 
is spectrally stable to same period perturbations. Note moreover that 
\begin{equation}\label{E:B'}
\mathbf{B}_{\tau,0}=\left(\begin{matrix}
\frac{i}{2}(\omega_{1,\tau}+\omega_{-1,\tau}) & \frac{1}{2}(\omega_{1,\tau}-\omega_{-1,\tau}) & 0 \\
-\frac{1}{2}(\omega_{1,\tau}-\omega_{-1,\tau}) & \frac{i}{2}(\omega_{1,\tau}+\omega_{-1,\tau}) & 0 \\
0 & 0 & i\omega_{0,\tau}\end{matrix}\right)
\end{equation}
for all $\tau \in [0,1/2]$ and $\mathbf{I}_{0}=\mathbf{I}$. 

\subsection{Proof of the main results}\label{sec:index}

We shall examine the roots of the characteristic polynomial 
\begin{equation}\label{def:P}
P(\mu;\tau,a)=\det({\bf B}_{\tau,a}-\mu{\bf I}_{a})
=c_3(\tau,a)\mu^3+ic_2(\tau,a)\mu^2+c_1(\tau,a)\mu+ic_0(\tau,a)
\end{equation}
for $\tau$ and $|a|$ sufficiently small
and derive a modulational instability index for \eqref{E:whitham}. 
As discussed above, these roots coincide with the eigenvalues of \eqref{E:Ltau} bifurcating from zero for $\tau$ and $|a|$ sufficiently small.

\

Note that $c_j$'s, $j=0,1,2,3$, depend smoothly upon $\tau$ and $a$ for $\tau, |a|$ sufficiently small. 
Since $\sigma(\L_{\tau,a})$ is symmetric about the imaginary axis, moreover, $c_j$'s are real. 
Furthermore \eqref{E:sym} implies that $c_3, c_1$ are even functions of $\tau$
while $c_2, c_0$ are odd and that $c_j$'s, $j=0,1,2,3$, are even in $a$.
Clearly, $\mu=0$ is a root of $P(\cdot\,;0,a)$ with multiplicity three
and $\tau=0$ is a root of $P(0;\cdot, a)$ with multiplicity three for $|a|$ sufficiently small. 
Therefore $c_j(\tau,a)=O(\tau^{3-j})$ as $\tau\to 0$, $j=0,1,2,3$, for $|a|$ sufficiently small.
Consequently we may write that 
\begin{equation}\label{def:dj}
c_j(\tau,a)=d_j(\tau,a)\tau^{3-j}, \qquad j=0,1,2,3,
\end{equation}
where $d_j$'s depend smoothly upon $\tau$ and $a$, and they are real and even in $a$ 
for $\tau, |a|$ sufficiently small. We may further write that
\begin{align*}
P(-i\tau\lambda;\tau,a)=i\tau^3(d_3(\tau,a)\lambda^3
-d_2(\tau,a)\lambda^2-d_1(\tau,a)\lambda+d_0(\tau,a)).
\end{align*}
The underlying, periodic traveling wave is then modulationally unstable
if the characteristic polynomial $P(-i\tau\cdot\,;\tau, a)$ admits a pair of complex roots, i.e., 
\begin{equation}\label{def:Delta}
\Delta_{\tau,a}:=\left(18d_3d_2d_1d_0+d_2^2d_1^2+4d_2^3d_0+4d_3d_1^3-27d_3^2d_0^2\right)(\tau,a)<0
\end{equation}
for $\tau$ and $|a|$ sufficiently small, 
where $d_j$'s, $j=0,1,2,3$, are in \eqref{def:P} and \eqref{def:dj},
while it is spectral stable near the origin
if the polynomial admits three real roots, i.e., if $\Delta_{\tau,a}>0$. 

\

Since $d_j$'s, $j=0,1,2,3$, are even in $a$, 
we may expand the discriminant in \eqref{def:Delta} and write that 
\[
\Delta_{\tau,a}=\Delta_{\tau,0}+\Lambda(\k) a^2 +O(a^2(a^2+\tau^2))
\]
as $\tau, |a| \to 0$. 
Since the zero solution of \eqref{E:whitham} is spectrally stable (see Section~\ref{sec:spec}),
furthermore, $\Delta_{\tau,0}\geq 0$ for all $\tau \in [0,1/2]$. 
Indeed we use \eqref{E:B} to calculate that $\Delta_{0,0}=0$ and that
\[
\Delta_{\tau,0}=\frac{\left(\omega_{0,\tau}-\omega_{1,\tau}\right)^2\left(\omega_{0,\tau}-\omega_{-1,\tau}\right)^2\left(\omega_{1,\tau}-\omega_{-1,\tau}\right)^2}{\tau^6}>0
\]
for all $\tau \in (0,1/2]$. 
Therefore the sign of $\Lambda(\k)$ decides 
the modulational instability versus spectral stability near the origin of the underlying wave. 
As a matter of fact, if $\Lambda(\k)<0$ then $\Delta_{\tau,a}<0$ for $\tau$ sufficiently small 
for $a>0$ sufficiently small but fixed, implying modulational instability, whereas 
if $\Lambda(\k)>0$ then $\Delta_{\tau,a}>0$ for all $\tau$ and $|a|$ sufficiently small,
implying spectral stability. 
We then use Mathematica and calculate \eqref{E:B} and \eqref{E:I} to ultimately find that
\begin{equation}\label{def:Lambda}
\Lambda(\k)=\frac{2(2\k\alpha'(\k)+\k^2\alpha''(\k))(\alpha(\k)-1+\k\alpha'(\k))^3(3\alpha(\k)-2\alpha(2\k)-1+\k\alpha'(\k))}{\alpha(\k)-\alpha(2\k)}.
\end{equation}

\

A straight forward calculation reveals that
$\k\mapsto\k(\alpha(\k)-1)$ is strictly decreasing and concave down over the interval $(0,\infty)$, 
whence
\[
(\k(\alpha(\k)-1))' = \alpha(\k)-1+\k\alpha'(\k)<0
\]
and
\[
(\k(\alpha(\k)-1))''= 2\alpha'(\k)+\k\alpha''(\k)<0
\]
for $0<\k<\infty$.  Therefore the sign of $\Lambda(\k)$ agrees with that of
\begin{align*}
\Gamma(\k):=&3\alpha(\k)-2\alpha(2\k)-1+\k\alpha'(\k) \\
=&2(\alpha(\k)-\alpha(2\k))+\frac{d}{d\k}(\k(\alpha(\k)-1)).
\end{align*}
Furthermore, the un-normalization of \eqref{E:whitham} (see \eqref{E:normalization})
merely changes $\k$ to $d\k$ above, and has no effects regarding stability upon the scaling of $t$ and $u$. 
This completes the proof of Theorem \ref{thm:main}.

\

To proceed, a straightforward calculation reveals that 
\begin{equation}\label{E:LambdaLimit}
\lim_{\k\to 0^+}\frac{\Gamma(\k)}{\k^2}=\frac{1}{2}\quad\textrm{and}\quad
\lim_{\k\to\infty}\Gamma(\k)=-1,
\end{equation}
whence the intermediate value theorem dictates at least one transversal root of $\Gamma$, 
corresponding to change in stability. 
Therefore a small-amplitude, $2\pi/\k$-periodic traveling wave of \eqref{E:whitham} 
is modulationally unstable if $\k>0$ is sufficiently large, while 
it is spectrally stable if $\k>0$ is sufficiently small.

Furthermore a numerical evaluation reveals that $\Gamma(\k)$ takes at least one transversal root 
$\k_c\approx 1.146$ for $0<\k<\infty$ (see Figure~\ref{F:GammaFig}) 
and $\Gamma(\k)<0$ for $\k>\k_c$ (see Figure~\ref{F:GammaInvertFig}).
Consequently $\Gamma(\k)$ takes a unique root $\k_c$ for $0<\k<\infty$, 
corresponding to exactly one switch in stability.
The un-normalization (see \eqref{E:normalization}) then completes the proof of Corollary~\ref{cor:main}.

\section{Extensions}\label{sec:extension}

The results in Section~\ref{sec:existence} and Section~\ref{sec:stability} 
are readily adapted to related, nonlinear dispersive equations. 
We shall illustrate this by deriving a modulational instability index 
for small-amplitude, periodic traveling waves of equations of Whitham type 
\begin{equation}\label{E:equation}
\partial_tu+\mathcal{M}\partial_xu+\partial_x(u^2)=0.
\end{equation}
Here $\mathcal{M}$ is a Fourier multiplier defined, abusing notation, as
$\widehat{\mathcal{M}f}(\xi)=\alpha(\xi)\hat{f}(\xi)$, allowing for nonlocal dispersion.
We assume that $\alpha$ is even and strictly decreasing\footnote{
More precisely, we assume that $\alpha(\k)-\alpha(n\k)=0$ has no solutions 
for $\k>0$ for $n=2,3,\dots$. Otherwise, a resonance occurs between 
the first and higher harmonics and our theory does not apply.}
over the interval $(0,\infty)$, and satisfies $\alpha(0)=1$. Clearly \eqref{def:M} fits into the framework. 

\

Following the arguments in Section~\ref{sec:existence} 
we obtain for each $\k>0$ a smooth, two-parameter family of 
small-amplitude, $2\pi/\k$-periodic traveling waves $w=w(\k,a,b)$ of \eqref{E:equation},
propagating at the speed $c=c(\k,a,b)$, where $|a|$ and $|b|$ are sufficiently small. 
Furthermore $w$ and $c$ satisfy, respectively, \eqref{E:w(k,a,b)} and \eqref{E:c(k,a,b)}. 

Following the arguments in Section~\ref{sec:stability}
we then derive the modulational instability index $\Lambda(\k)$ in \eqref{def:Lambda}
deciding the modulational instability versus spectral stability
of small-amplitude, periodic traveling waves of \eqref{E:equation} near the origin. 
We merely pause to remark that the derivation of the index does not use 
specific properties of the dispersion symbol. 

\

For example, consider the KdV equation with fractional dispersion (fKdV)
\begin{equation}\label{E:fKdV}
\partial_tu+\partial_x u-(\sqrt{-\partial_x^2})^\sigma \partial_xu+\partial_x(u^2)=0,\quad \sigma>1/2,
\end{equation}
for which $\alpha(\xi;\sigma)=1-|\xi|^\sigma$. 
In the case of $\sigma=2$, notably, \eqref{E:fKdV} recovers the KdV equation 
while in the case of $\sigma=1$ it is the Benjamin-Ono equation. 
In the case of $\sigma=-1/2$, moreover, \eqref{E:fKdV} was argued in \cite{Hur-breaking}
to approximate up to quadratic order the surface water wave problem in two dimensions
in the infinite depth case. Note that \eqref{E:fKdV} is nonlocal for $\sigma<2$. 
Note moreover that $\alpha$ is homogeneous. 
Hence one may take without loss of generality that $\k=1$, and
the modulational instability index depends merely upon the dispersion exponent $\sigma$. 
As a matter of fact, we substitute the dispersion relation into \eqref{def:Lambda} and calculate that
\begin{equation}\label{def:LKdV}
\Lambda_{fKdV}(\k;\sigma)=\frac{2\k^{4\sigma}\sigma(1+\sigma)^4(2^{\sigma+1}-3-\sigma)}{2^\sigma-1}.
\end{equation}
The modulational instability index $\Lambda_{fKdV}$ is negative (for all $\k>0$) if $\sigma<1$, 
implying modulational instability, 
whereas it is positive if $\sigma>1$, implying spectral stability to square integrable perturbations.
The results reproduce those in \cite{J2013}, 
which addresses a general, power-law nonlinearity. 
Observe that the index \eqref{def:LKdV} vanishes if $\sigma=1$, which is inconclusive. 
It was shown in \cite{BHV}, on the other hand, that all periodic traveling waves 
of the Benjamin-Ono equation are spectrally stable near the origin.

\

Another interesting example is the intermediate long-wave (ILW) equation
\begin{equation}\label{E:ILW}
\partial_tu+\partial_xu+(1/H)\partial_xu-\mathcal{N}_H\partial_xu+\partial_x(u^2)=0,
\end{equation}
which was introduced in \cite{Joseph} to model nonlinear dispersive waves 
at the interface between two fluids of different densities, contained at rest in a channel,
where the lighter fluid is above the heavier fluid;
$H>0$ is a constant and $\mathcal{N}_H$ is a Fourier multiplier, defined as
\[
\widehat{\mathcal{N}_H f}(\xi)=\xi\coth(\xi H)\widehat{f}(\xi).
\]
Incidentally $\mathcal{N}_H$ is the Dirichlet-to-Neumann operator for a strip of size $H$
with a Dirichlet boundary condition at the other boundary.
Clearly \eqref{E:ILW} is of the form \eqref{E:equation}, for which $\alpha(\k;H)=1+1/H-\k\coth(\k H)$.
We may substitute into \eqref{def:Lambda} and calculate that 
\begin{multline}\label{def:LILW}
\Lambda_{ILW}(\k;H)=\frac{((-1+4H^2\k^2)\cosh(H\k)+\cosh(3H\k)-8H\k\sinh(H\k))^2}
{32H^4\sinh(H\k)^{12}}\\
\times(1-2H^2\k^2-\cosh(2H\k)+2H\k\sinh(2H\k))^3.
\end{multline}
Note that the sign of $\Lambda_{ILW}(\k;H)$ is the same as that of 
\[
\Gamma_{ILW}(z)=1-2z^2-\cosh(2z)+2z\sinh(2z)
\]
evaluated at $z=\k H$.
An examination of the graph of $\Gamma_{ILW}$ furthermore
indicates that it is positive over the interval $(0,\infty)$. 
We therefore conclude that for each $H>0$ small-amplitude, periodic traveling waves of \eqref{E:ILW} 
are spectrally stable to square integrable perturbations near the origin, 
rigorously justifying a formal calculation in \cite{Pel}, for instance, using amplitude equations.
To the best of the authors' knowledge, this is the first rigorous result
about the modulational stability and instability for the intermediate long-wave equation.

Note moreover that $\Lambda_{ILW}(\k;H)\to 0$ as $\k H\to\infty$ for any $\k>0$,
while \eqref{E:ILW} recovers the Benjamin-Ono equation as $H \to \infty$, 
corresponding to \eqref{E:fKdV} with $\sigma=1$.
As a matter of fact, \eqref{E:ILW} bridges the KdV equation ($H\to 0$) 
and the Benjamin-Ono equation ($H\to \infty$).
Incidentally $\Lambda_{fKdV}(\k,\sigma)$ in \eqref{def:LKdV} vanishes identically in $\k$ at $\sigma=1$. 

\subsection*{Acknowledgements}
The authors thank the referees for their careful reading of the manuscript 
and for their many useful suggestions and references.  
VMH is supported by the National Science Foundation under grant DMS-1008885 
and by an Alfred P. Sloan Foundation fellowship. 
MJ gratefully acknowledges support from the National Science Foundation under grant 
DMS-1211183.

\bibliographystyle{amsalpha}
\bibliography{stabilityBib}

\end{document}